\documentclass[11pt,a4paper]{amsart}
\usepackage{amsmath,amsfonts,amsthm,amsopn,color,amssymb,enumitem}
\usepackage[a4paper]{geometry}
\usepackage{palatino}
\usepackage{graphicx}
\usepackage[colorlinks=true]{hyperref}
\hypersetup{urlcolor=blue, citecolor=red, linkcolor=blue}

\usepackage{cite}
\usepackage{relsize}
\usepackage{esint}
\usepackage{verbatim}
\usepackage{mathrsfs}
\usepackage{xcolor}
\usepackage{tikz}

\newcommand{\e}{\varepsilon}

\newcommand{\ee}{{\mathtt e}}

\renewcommand{\d }{\delta }

\renewcommand{\O}{\mathcal{O}}

\newcommand{\Ua}{{\mathcal{U}}}

\newcommand{\Uj}{{\mathcal{U}_{\delta_j, \xi_j}}}
\newcommand{\Ui}{{\mathcal{U}_{\delta_i, \xi_i}}}
\newcommand{\Ul}{{\mathcal{U}_{\delta_l, \xi_l}}}
\newcommand{\UM}{{\mathcal{U}_{\delta_m, \xi_m}}}
\newcommand{\U}{{\mathcal{U}_{\delta, \xi}}}

\newcommand{\Uh}{{\mathcal{U}_{\delta_h, \xi_h}}}

\newcommand{\beq}{\begin{equation}}
\newcommand{\eeq}{\end{equation}}

\newtheorem{theorem}{Theorem}[section]
\newtheorem*{theorem*}{Theorem}
\newtheorem{lemma}[theorem]{Lemma}

\newtheorem{definition}[theorem]{Definition}
\newtheorem{proposition}[theorem]{Proposition}

\theoremstyle{definition}

\renewcommand{\(}{\left(}
\renewcommand{\)}{\right)}

\begin{document}

\title[Multi-bubble solutions for the Brezis-Nirenberg problem in  four dimensions]{Multi-bubble solutions for the Brezis-Nirenberg problem in  four dimensions}

\author[A. Pistoia]{Angela Pistoia}
\address{\noindent  Dipartimento di Scienze di Base e Applicate per l'Ingegneria, Università degli Studi di Roma Sapienza}
\email{angela.pistoia@uniroma1.it}

\author[G. M. Rago]{Giuseppe Mario Rago}
\address{\noindent  Dipartimento di Matematica, Universit\'a degli Studi di Bari Aldo Moro,Italy }
\email{g.rago6@phd.uniba.it}

\author[G. Vaira]{Giusi Vaira}
\address{\noindent  Dipartimento di Matematica, Universit\'a degli Studi di Bari Aldo Moro,Italy }
\email{giusi.vaira@uniba.it}
\thanks{Work partially supported by 
the MUR-PRIN-P2022YFAJH ``Linear and Nonlinear PDE’s: New directions and Applications" and by the INdAM-GNAMPA project ``Fenomeni non lineari: problemi locali e non locali e loro applicazioni" CUP E5324001950001.}

\subjclass{35B44, 35B33, 35J25}
\keywords{Brezis-Nirenberg problem; blowing-up solutions; Lyapunov-Schmidt reduction.}

\maketitle
\begin{abstract}
The paper addresses the existence of multi-bubble solutions for the well-known Brezis–Nirenberg problem.
Although there is extensive literature on the subject, the existence of solutions that blow up at multiple points in a 4D bounded domain remains an open problem.
The goal of the present paper is to resolve this longstanding issue. In particular, we exhibit examples of domains where a large number of multi-bubble solutions exist.
Our result can also be seen as the counterpart of the asymptotic analysis carried out by König and Laurin in
{\em Ann. Inst. H. Poincar\'e C Anal. Non Lin\'eaire, 2024}.
\end{abstract}
\section{Introduction and main results}
\noindent Let us consider the well-known 
Brezis-Nirenberg problem, i.e.   
\begin{equation}\label{pb}
\begin{cases}
- \Delta u = |u|^{2^*-2}u + \varepsilon u \ \ \ \ \ \ \ \ \ \ \ \ \ \  & \mbox{in}\ \ \Omega, \\
u= 0 \ \ \ \ \ \ \ \ \ \ \ \ \ \ \ \ & \mbox{on} \ \ \partial \Omega \ \\ 
\end{cases}
\end{equation} \\ 
where $\Omega$ is a bounded regular domain in $\mathbb{R}^{N}$ with $N\geq 3$, $\varepsilon$ is a positive parameter and $2^*:=\frac{2N}{N-2}$ is the critical exponent for the Sobolev embedding. \\
Problem \eqref{pb} has been  introduced by Brezis and Nirenberg in their celebrated paper \cite{BN}.\\
A particular feature of \eqref{pb}, due to the critical behaviour of the nonlinearity, is the possible existence of solutions which blow-up at one or more points in the domain as the parameter $\e$ approaches $0$ in dimension $N\geq 4$ or the parameter $\e$ approaches a positive number $\lambda_*$ in dimension $N=3$. The description
of the profile of the positive blowing-up solutions has been the subject of a wide
literature. Starting from the pioneering paper by Brezis and Peletier \cite{BP} where
the authors consider the radial case, a lot of results have been obtained concerning the asymptotic profiles of solutions to \eqref{pb}. Firstly, Han \cite{H} and Rey \cite{R} studied the profile of one-peak solutions in a
general domain. Most recently, a fine multi-bubbles analysis has been object of two papers by K\"onig and Laurin. Indeed, in \cite{KL1} the authors give an accurate description of the
blow-up profile of a solution with multiple blow-up points when $N\geq 4$ while in \cite{KL2} they consider the case $N=3$. The  3-dimensional case was firstly faced by Druet \cite{D}. \\ In particular, the asymptotic analysis of blowing-up solutions ensures that the
blow-up points are nothing but the critical points of a suitable function which
involves the Green’s function of $-\Delta$ in $\Omega$  with Dirichlet boundary
condition and the Robin’s function.
While the behavior of blowing-up positive solutions is by now clear, the profile
of the sign-changing solutions is far being completely understood. As far as we
know the complete scenario of sign-changing solutions has been obtained only in
the radial case by Esposito, Ghoussoub, Pistoia and Vaira in \cite{EGPV} in dimensions $N\geq 7$
and by Amadori, Gladiali, Grossi, Pistoia and Vaira \cite{AGGPV} in lower dimensions.\\

Another interesting field of research is to find a   solution to \eqref{pb} which blows up at one or more points whose profile is the one predicted by the asymptotic analysis.\\ The first result is due to Rey that in \cite{R} builds a solution that blows-up at one non-degenerate critical point of the Robin function when $N\geq 5$. In \cite{MP}, Musso and Pistoia construct positive solutions with many blow-up points when $N\geq 5$ while Musso and Salazar \cite{MS} consider the case of multiple blow-up points in dimension $N=3$. Sign-changing solutions have been built by Iacopetti and Vaira  \cite{IV1, IV2}, Liu, Vaira, Wei and Wu \cite{LVWW}, Morabito, Pistoia and Vaira in \cite{MPV}, Pistoia and Vaira in \cite{PV,PV1} and Premoselli in \cite{P}.
\\

Although there is a huge literature on \eqref{pb}, many problems are still open. 
Among others, one is concerning the existence of positive blowing-up  solutions  in dimension $N=4$. 
Remarkably, differently from   solutions with only one blow-up point  (see for example Pistoia and Rocci \cite{PR}) the  case with several blow-up points has not been studied
 in the literature, yet. The goal of the present paper is to close
 this gap. \\

 Let us introduce the objects necessary to state our result.
 Let $G(x, y)$ be the Green function of the Laplace operator with Dirichlet boundary conditions and $H$ be its regular part, i.e. 
 $$G(x, y):=\frac{1}{2\omega}\frac{1}{|x-y|^{2}}-H(x, y)$$ where $\omega:=2\pi^2$ denotes the surface area of the unit sphere in $\mathbb R^4$ and 
$H(x, y)$, $x, y\in\Omega$, satisfies $$-\Delta H(x, y)=0\ \hbox{in}\ \Omega,\ H(x, y)=\frac{1}{|x-y|^{2}}\ \hbox{on}\  \partial\Omega.$$  
For every $x\in\Omega$ the leading term of the regular part of the Green's function $$\tau_\Omega(x)=H(x,x)$$ is called the {\it Robin function}  of $\Omega$ at the point $x$.\\
Next, given an integer $k\geq 1$ we define
$$\Omega^*_k:=\left\{\boldsymbol\xi:=(\xi_1, \ldots, \xi_k)\in \Omega^k\,\,:\,\, \xi_i\neq \xi_j \,\, \mbox{for}\,\, i\neq j\right\}.$$
and for any $\boldsymbol\xi\in\Omega^*_k$ we define the symmetric matrix $M(\boldsymbol\xi)=\(m_{ij}(\boldsymbol\xi)\)_{1\leq i, j\leq k}$ whose components are given by 
\begin{equation}\label{Mxi}m_{ij}(\boldsymbol\xi):=\left\{\begin{aligned}&\tau_\Omega(\xi_i)\quad \ \ \ \ &\mbox{if}\,\, i=j\\
&-G(\xi_i, \xi_j)\quad &\mbox{if}\,\, i\neq j.\\
\end{aligned}\right.\end{equation}
We denote by $\Lambda_1( \boldsymbol\xi)$ the smallest eigenvalue of $M(\boldsymbol\xi)$.
Using the Perron-Frobenius theorem it is easy to show (see \cite{BLR}) that the eigenvalue $\Lambda_1(\boldsymbol\xi)$ is simple and the corresponding eigenvector can be chosen to
have strictly positive components. We denote by $\boldsymbol\ee(\boldsymbol\xi):=(\ee_1(\boldsymbol\xi), \ldots, \ee_4(\boldsymbol\xi))^T\in\mathbb R^4$ the unique vector such that
\begin{equation}\label{impM}M(\boldsymbol\xi)\boldsymbol\ee(\boldsymbol\xi)=\Lambda_1(\boldsymbol\xi)\boldsymbol\ee(\boldsymbol\xi)\ \hbox{and}\  \boldsymbol\ee_1(\boldsymbol\xi)=1.\end{equation}
\vskip0.2in
Finally, we remind the definition of {\em stable critical set} introduced by Y.Y. Li in \cite{yyl}. 
\begin{definition}\label{yy1}
Given a smooth function $f:D\subset\mathbb R^n\to\mathbb R,$ a set $\mathscr K$ of critical point of $f$ is    stable if there exists a neighbourhood $\Theta$ of $\mathscr K$ such that the Brouwer degree  $\mathtt{deg}(\nabla f,\Theta,0)\not=0,$ where 
$\mathtt{deg}$ denotes  the Brouwer degree. \end{definition}
Examples of stable critical sets are listed below.
\begin{itemize}
\item  $\mathscr K$ is a strict local minimum (or maximum) set of $f,$ i.e. $f(x)=f(y)$ for any $x,y\in \mathscr K$ and 
$f(x)< f(y)$ (or $f(x)> f(y)$) for any $x\in \mathscr K$ and $y\in \Theta\setminus\mathscr K$
\item $\mathscr K=\{x_0\}$ is an isolated critical point of $f$ with $\mathtt{deg}(\nabla f,B(x_0,\rho),0)\not=0$  for some small $\rho>0$ (e.g. $x_0$ is a non-degenerate critical point of $f$).
\end{itemize}

Our result is the following.
\begin{theorem}\label{main}
Let $\mathscr K$ be a  stable critical set of $\Lambda_1$ with $\Lambda_1(\boldsymbol\xi)>0$ for any $\boldsymbol \xi\in \mathscr K.$ Then, there exists a family of  solutions of 
\begin{equation}\label{pb4}-\Delta u=u^3+\e  u\ \hbox{in}\ \Omega, \  u>0\ \hbox{in}\ \Omega,\  u=0\ \hbox{on}\ \partial\Omega
\end{equation} which blows-up  at the points $\xi_{1_0}, \ldots, \xi_{k_0}$ with rates of concentration $\delta_{1_\e},\ldots,\delta_{k_\e}$ such that $\boldsymbol\xi_0:=(\xi_{1_0}, \ldots, \xi_{k_0})\in \mathscr K$ and
\begin{equation}\label{k1}
\e \log\delta_{i, \e}^{-1}\to \Lambda_1(\boldsymbol\xi_0),\quad\mbox{as}\,\,\ \e\to0.\end{equation}\end{theorem}

The result is new and it can be seen as the counterpart of the asymptotical analysis performed by  K\"onig and Laurin's result in \cite{KL1}. Indeed in \cite[Theorem 1.1]
{KL1} the authors proved that if $u_\varepsilon$ is a solution to \eqref{pb4} which blows-up at $k$ points $\xi_{1_0}, \ldots, \xi_{k_0}$, then 
the rates of concentration  $\delta_{1_\e},\ldots,\delta_{k_\e}$ must satisfy \eqref{k1} and $\boldsymbol\xi_0:=\(\xi_{1_0}, \ldots, \xi_{k_0}\)\in\Omega_k^*$
and $\boldsymbol d:=\(d_1,\dots,d_k\)\in\mathbb R^k_+$, with  $d_i:=\lim\limits_{\e\to0}
 {\delta_{i_\e}\over\delta_{1_\e}} $,  must be a critical point of the function
\begin{equation}\label{k2}  (\boldsymbol d,\boldsymbol\xi)\longrightarrow \boldsymbol dM(\boldsymbol\xi)\boldsymbol d^T-\frac1{8\pi^2}\Lambda_1(\boldsymbol\xi_0)\boldsymbol d\boldsymbol d^T.\end{equation}

The proof of our result relies on a classical Ljapunov-Schmidt procedure, which is described in Section \ref{sec1}. The exponential decay of the concentration rates (see \eqref{k1})
makes it harder to derive  the so-called reduced problem \eqref{sys2}, which at a first glance does not have a variational structure.
In fact, the major difference with the higher dimensional case  (see \cite{MP}) and the main novelty of this paper lie in the study of the reduced problem, which is  carried out in  Section \ref{sec2}. In Section \ref{sec3} we show a couple of examples for which multipeaks solutions do exist. The first example is the dumbbell-shaped domain constructed by joining $k$ disjoint subdomains via thin necks. The second example is a thin annulus for which we prove the existence of solutions concentrating at two and four points. The restriction on the number of blowing-up points is merely due to technical reasons, indeed we strongly believe that the number of peaks grows as the thickness of the annulus decreases. (see \eqref{CON}).
All the technical estimates needed in  the reduction procedure are contained in  Appendix \ref{app}.

\section{The reduction procedure}\label{sec1}

\subsection{Preliminaries}
Let $H^{1}_{0}(\Omega)$ be the Hilbert space equipped with the usual inner product and the usual norm 
\begin{equation*}
\langle u,v \rangle := \int_{\Omega} \nabla u \cdot \nabla v ,\quad \| u \|  := \left( \int_{\Omega} | \nabla u |^{2} \right)^{\frac{1}{2}}.
\end{equation*}
For $r \in [1, +\infty)$ the space $L^{r}(\Omega)$ is  equipped with the standard norm 
\begin{equation*}
\| u \|_{r} = \left( \int_{\Omega} |u|^{r} \right)^{\frac{1}{r}}.
\end{equation*}

As usual, let $i^{*}_\Omega: L^{\frac{4}{3}}(\Omega) \to H^{1}_{0}(\Omega)$ be the adjoint operator of the embedding $i_\Omega : H^{1}_{0}(\Omega) \to L^{4}(\Omega)$, i.e. 
$$u=i^{*}_\Omega(f)\ \hbox{if and only if}\ 
\langle u, f \rangle = \int_{\Omega} f(x) \varphi(x) \ dx
\ \hbox{
for all}\ \varphi \in H^{1}_{0}(\Omega).
$$
The operator $i^{*}_\Omega : L^{\frac{4}{3}}(\Omega) \to H^{1}_{0}(\Omega)$ is continuous as 
\begin{equation*}
\| i^{*}_\Omega(f) \|_{H^{1}_{0}(\Omega)} \leq S^{-1} \| f \|_{\frac{4}{3}}
\end{equation*}
where $S$ is the best constant for Sobolev embedding.\\
Therefore, the  problem \eqref{pb4} can be rewritten as 
\begin{equation}\label{213}
u=i^{*}_\Omega(f(u)+ \varepsilon u),\ u \in H^{1}_{0}(\Omega),
\end{equation}
where $f(u)=(u^+)^3.$ \\
Next, we introduce the well-known Aubin-Talenti bubbles  
\begin{eqnarray*}
\mathcal{U}_{\delta,\xi}(x)=\alpha_4\left(\frac{\delta}{\delta^{2}+|x-\xi|^2}\right)=\delta^{-1}\mathcal{U}_{1,0}\left(\frac{x-\xi}{\delta}\right),\ \delta>0,\ x,\xi\in\mathbb R^4,\  
\end{eqnarray*}
with $\alpha_4:=2\sqrt 2$. We denote by $\Ua:=\mathcal{U}_{1,0}$.
They are all the positive solutions of (see \cite{A}, \cite{CGS}, \cite{T}) the critical problem
\begin{equation}\label{pblim}
 -\Delta\, \mathcal{U}_{\delta,\xi}=\mathcal{U}_{\delta,\xi}^{3}\ \hbox{in}\ \mathbb R^4.\end{equation}

Given a point $\xi\in\Omega,$ we denote by $P\mathcal{U}_{\delta,\xi}$ the projection of $\mathcal{U}_{\delta,\xi}$ into $H^1_0(\Omega)$, i.e. the unique solution  of 
\begin{eqnarray}\label{Pre0001}
 -\Delta\, P\mathcal{U}_{\delta,\xi}=\mathcal{U}_{\mu,\xi}^{3}\ \hbox{in}\ \Omega,\ 
 P\mathcal{U}_{\delta,\xi}=0\ \hbox{on}\ \partial\Omega.
\end{eqnarray}
\\  It is well known that the following expansion holds (see \cite{R})
\begin{equation}\label{exp}
P\U:=\U-\mathfrak C\d H(x, \xi)+\mathcal O\left(\d^{3}\right)\quad \mbox{as}\, \d\to 0\end{equation}
uniformly with respect to $x\in\Omega$ and $\xi$ in compact sets of $\Omega$ and 
\begin{equation}\label{exp1}
P\U:=\mathfrak C\d G(x, \xi)+\mathcal O\left(\d^{3}\right)\quad \mbox{as}\, \d\to 0\end{equation}
uniformly with respect to $x$ in compact sets of $\Omega\setminus\{\xi\}$ and $\xi$ in compact sets of $\Omega$ where $\mathfrak C:=2\alpha_4\omega=8\sqrt2\pi^2$.\\  

\subsection{The ansatz} 
Let $k \geq 1$ be a fixed integer. We look for a solution to \eqref{213} of the form 
\begin{equation}\label{soluzione}
u_\e = \sum_{i=1}^{k} P\Ui+\phi_\e, 
\end{equation}
where the blow-up rates are $\delta_{i}=\delta_{i}(\varepsilon) $ are choosen as
\begin{equation}\label{delta}
\delta_1:=e^{-\frac{8\pi^2 \lambda}{ \e}},\ \d_i:=e^{-\frac{8\pi^2 \lambda}{ \e}}d_i\ \hbox{with}\ \lambda>0\ \hbox{and}\
 d_i>0\ \hbox{for}\  i=2, \ldots, k
\end{equation}  
and the blow-up points are $\xi_{i}$ belong to the set
\begin{equation*}
D_\rho = \{ \pmb{\xi} \in \Omega^{*}_k\,\,:\,\, \mbox{dist}(\xi_{i}, \partial \Omega) \geq 2\rho, |\xi_{i} - \xi_{j}| \geq 2\rho,\,\, \forall i,j = 1, \cdots, k, i \neq j \}.
\end{equation*} for some $\rho>0$ small.\
In the following, we agree that $\boldsymbol \delta:=(\delta_1,\dots,\delta_k)$ and $\boldsymbol \xi:=(\xi_1,\dots,\xi_k).$\\
The higher order term  $\phi_\e$ belongs to the space  
\begin{equation*}
\mathcal K^{\perp}_{\pmb{\d}, \pmb{\xi}} =\left \{ \phi \in H^{1}_{0}(\Omega)\,\, :\,\,\ \langle \phi, \psi \rangle = 0\ \forall\ \psi\in \mathcal K_{\pmb{\d}, \pmb{\xi}}\right\},
\end{equation*}
where
\begin{equation*}
\mathcal K_{\pmb{\d}, \pmb{\xi}} = \mbox{span} \{ P \psi_{\delta_{i}, \xi_{i}}^{j}\,\, :\,\, i= 1, \cdots, k, j=0, \cdots, 4\}.
\end{equation*}
Here $P\psi_{\delta_{i}, \xi_{i}}^{j}$ are the orthogonal projections onto $H^1_0(\Omega)$ of the functions
\begin{equation*}
\psi_{\delta_{i}, \xi_{i}}^{j}(x) = \frac{1}{\delta_{i}} \psi^{j} \left(\frac{x-\xi_{i}}{\delta_{i}} \right),\   j=0, \ldots, 4,\ i=1, \ldots, k,
\end{equation*}
where 
\begin{equation*}
\psi^{0}(x) := \Ua(x) + \nabla \Ua(x) \cdot x  = \alpha_4 \frac{|x|^{2}-1}{(1+|x|^{2})^{2}}
\end{equation*}
and 
\begin{equation*}
\psi^{j}(x) : = \frac{\partial \Ua}{\partial x_{j}}(x) = -2\alpha_4 \frac{x_{j}}{(1+|x|^{2})^{2}}, \  j=1, \cdots, 4.
\end{equation*}
generate the space of solutions of the linear equation
\begin{equation*}
-\Delta \psi = 3\Ua^{2} \psi \ \hbox{in} \   \mathbb{R}^{4}.
\end{equation*}
It is useful to recall the well-known expansions
\begin{equation}\label{expderi}
P \psi^{0}_{\delta_{i}, \xi_{i}} =\psi^{0}_{\delta_{i}, \xi_{i}} -\mathfrak C\delta_i H(x, \xi_i)+\mathcal O(\delta_i^2)\end{equation}
and for all $j=1, \ldots, 4$
\begin{equation}\label{expderi1}
P \psi^{j}_{\delta_{i}, \xi_{i}} =\psi^{j}_{\delta_{i}, \xi_{i}} - \mathfrak C^2\delta_i^2\partial_{\xi_j} H(x, \xi_i)+\mathcal O(\delta_i^3).\end{equation}
\subsection{An equivalent system}
Let us introduce the linear projection $\Pi_{\pmb{\d}, \pmb{\xi}} : \mathcal K_{\pmb{\d}, \pmb{\xi}} \to \mathcal K_{\pmb{\d}, \pmb{\xi}}$ and $\Pi^{\perp}_{\pmb{\d}, \pmb{\xi}} : \mathcal K^{\perp}_{\pmb{\d}, \pmb{\xi}} \to \mathcal K^{\perp}_{\pmb{\d}, \pmb{\xi}}$ which are defined by 
\begin{equation*}
\Pi_{\pmb{\d}, \pmb{\xi}}(\phi) = \sum_{i=1, \cdots, k \atop j=0, \cdots, 4} \langle \phi, P \psi_{i}^{j} \rangle P \psi_{i}^{j} \quad\mbox{and}\quad\Pi^{\perp}_{\pmb{\d}, \pmb{\xi}}(\phi) = \phi - \Pi_{\pmb{\d}, \pmb{\xi}}(\phi).
\end{equation*}
In oder to simplify the notations we let also $W_{\pmb{\d}, \pmb{\xi}}:=\sum_{i=1}^k P\Ui$. \\ Since we look for a solution of the form $W_{\pmb{\d}, \pmb{\xi}}+\phi_\e$ then
equation \eqref{213} can be rewritten as the following system of two equations 
\begin{equation}\label{10}
\Pi^{\perp}_{\pmb{\d}, \pmb{\xi}} [ \mathcal{L}_{\pmb{\d}, \pmb{\xi}}(\phi_\e) - \mathcal{E}_{\pmb{\d}, \pmb{\xi}} - \mathcal{N}_{\pmb{\d}, \pmb{\xi}}(\phi_\e)] = 0
\end{equation}
and
\begin{equation}\label{11}
\Pi_{\pmb{\d}, \pmb{\xi}} [\mathcal{L}_{\pmb{\d}, \pmb{\xi}}(\phi_\e) - \mathcal{E}_{\pmb{\d}, \pmb{\xi}} - \mathcal{N}_{\pmb{\d}, \pmb{\xi}}(\phi_\e)] = 0,
\end{equation}
where the linear operator $\mathcal{L}_{\pmb{\d}, \pmb{\xi}}$ is 
\begin{equation}\label{29}
\mathcal{L}_{\pmb{\d}, \pmb{\xi}}(\phi_\e) = \phi_\e - i^{*}_\Omega(3\phi_\e W_{\delta, \xi}^{2} + \varepsilon \phi_\e),
\end{equation}
the error term $\mathcal{E}_{\pmb{\d}, \pmb{\xi}}$ is 
\begin{equation}\label{30}
\mathcal{E}_{\pmb{\d}, \pmb{\xi}}= i^{*}_\Omega(W_{\pmb{\d}, \pmb{\xi}}^{3} + \varepsilon W_{\pmb{\d}, \pmb{\xi}}) - W_{\pmb{\d}, \pmb{\xi}}
\end{equation}
and the nonlinear term $\mathcal{N}_{\pmb{\d}, \pmb{\xi}}$ is 
\begin{equation}\label{31}
\mathcal{N}_{\pmb{\d}, \pmb{\xi}}(\phi_\e) = i^{*}_\Omega\left(f(W_{\pmb{\d}, \pmb{\xi}}+\phi)-f(W_{\pmb{\d}, \pmb{\xi}})-f'(W_{\pmb{\d}, \pmb{\xi}})\phi\right).
\end{equation}

\subsection{Solving equation (\ref{10})}
The solvability of \eqref{10} in terms of $(\boldsymbol\xi,\boldsymbol\delta)$ is the first step in the Ljapunov-Schmidt procedure and follows by the proposition below.
\begin{proposition}\label{fixedpoint}
For any $\rho>0$ there exist $C>0$ and $\varepsilon_{0}>0$ such that for any $\varepsilon, \d_i \in (0, \varepsilon_{0})$ and any $\xi\in D_\rho$ there exists a unique $\phi=\phi_{\pmb{\d}, \pmb{\xi}}\in \mathcal K^{\perp}_{\pmb{\d}, \pmb{\xi}}$ solving \eqref{10} and satisfying 
\begin{equation}\label{phi}
\| \phi \| \leq C\( |\pmb{\d}|^2+\e|\pmb{\d}|\).\end{equation}
\end{proposition}
\begin{proof} The proof is standard and we refer to  \cite{MP,PR} for the details. In particular, estimate \eqref{phi} follows by the estimate of the error term $\mathcal{E}_{\pmb{\d}, \pmb{\xi}}$ whose proof is postponed in the Appendix.
\end{proof}

\subsection{Solving equation (\ref{11})}
Let $u_\e=W_{\pmb{\d}, \pmb{\xi}}+\phi$ where $\phi\in \mathcal K^\bot_{\pmb{\d}, \pmb{\xi}}$ is the function founded in 
Proposition \ref{fixedpoint}.   Equation \eqref{11} rewrites as
\begin{equation}\label{r1}
-\Delta u_\e -f(u_\e)-\e u_\e=\sum_{j=0, \ldots, 4\atop h=1, \ldots, k}\mathfrak c^i_h f'(\Ui)\psi^j_{\d_h, \xi_h}
\end{equation}
for some real numbers $\mathfrak c^i_h$ that depend on $\boldsymbol\delta$ and   $\boldsymbol \xi$. 
The second step  in the Ljapunov-Schmidt procedure
 consists in  finding  the rate parameters $\boldsymbol\delta$ and the points $\boldsymbol \xi$ so that all the $\mathfrak c^j_h$'s  in \eqref{r1} are zero and this is done by  solving a reduced problem as stated  in the Proposition below, whose proof is postponed in the Appendix.
\begin{proposition}\label{propro}
It holds true that (up to some constants)
$$\mathfrak c^0_h=\delta_h^2\tau_\Omega(\xi_h)-\sum_{i\neq h}\d_i\d_h G(\xi_i, \xi_h)+\frac{1}{4\omega}\e \d_h^2\ln\d_h+o(|\pmb\d|^2),\ h=1,\dots,k$$
and
$$\mathfrak c^i_h=\frac{\partial}{\partial(\xi_h)_l}\left[\d_h^2\tau_\Omega(\xi_h)-2\sum_{h\neq i}\d_i\d_hG(\xi_i,\xi_h)\right]+o(|\pmb\d|^2),\ i=1,\dots,4,\ h=1,\dots,k $$
as $\e \to 0$, uniformly  with respect to $(\boldsymbol\xi, \mathbf{d}, \lambda)$ (see \eqref{delta}) in compact sets of $D_\rho \times (0, +\infty)^{k-1}\times (0, +\infty).$
\end{proposition}

By Proposition \ref{fixedpoint} and  Proposition \ref{propro}, it follows that the problem reduces to find $\delta_i>0$ and $\xi_i\in\Omega$ for $i=1, \ldots, k$ such that 
\begin{equation}\label{sys1}\left\{\begin{aligned} &\left(\d_i \tau_\Omega(\xi_i)-\sum_{j\neq i}\d_j G(\xi_i, \xi_j)+\frac{1}{8\pi^2}\e\d_i \ln\d_i\right)\left(1+o(1)\right)=0\, \quad &i=1, \ldots, k\\\\
&\left(\d_i \frac{\partial}{\partial(\xi_i)_\ell} \tau_\Omega(\xi_i)-2\sum_{j\neq i}\d_j \frac{\partial}{\partial(\xi_i)_\ell}G(\xi_i, \xi_j)\right)\left(1+o(1)\right)=0\quad &i=1, \ldots, k.\end{aligned}\right.\end{equation}
More precisely, taking into account the choice of the $\delta_i$'s in \eqref{delta}, we have to find $\lambda>0,$ $d_2,\dots,d_{k}\in (0,+\infty)$ 
and $(\xi_1,\dots,\xi_k)\in D_\rho$ solution of the  system 
\begin{equation}\label{sys2}\left\{\begin{aligned} &\tau_\Omega(\xi_1)-\sum_{j=2}^k d_j G(\xi_1, \xi_j)-\lambda+o(1)=0\,\\
&d_i \tau_\Omega(\xi_i)-G(\xi_1, \xi_i)-\sum_{j=2\atop j\neq i}^k d_j G(\xi_i, \xi_j)-\lambda d_i+o(1)=0,\ i=2, \ldots, k\\
&\frac{\partial}{\partial(\xi_1)_\ell}\tau_\Omega(\xi_1) -2\sum_{j=2}^k d_j\frac{\partial}{\partial(\xi_1)_\ell} G(\xi_1, \xi_j)+o(1)=0,\ \ell=1, \ldots, 4\\
&d_i \frac{\partial}{\partial(\xi_i)_\ell} \tau_\Omega(\xi_i)-2\frac{\partial}{\partial(\xi_i)_\ell}G(\xi_1, \xi_i)-2\sum_{j=2\atop j\neq i}d_j \frac{\partial}{\partial(\xi_i)_\ell} G(\xi_i, \xi_j)+o(1)=0,\\ &\hskip9truecm i=2, \ldots, k,\   \ell=1, \ldots, 4.\end{aligned}\right.\end{equation}

\section{The reduced problem}\label{sec2}
\subsection{Proof of Theorem \ref{main}: completed}
We aim to solve system \eqref{sys2}. We set
$\mathbf{d}:=(d_2,\ldots\ldots, d_k)$ and we  introduce three different functions.
The first one $F_1: D_\rho \times (0, +\infty)^{k-1}\times (0, +\infty)\to\mathbb R$
is related to the first equation in \eqref{sys2}:
$$ F_1(\boldsymbol\xi, \mathbf{d}, \lambda)= \tau_\Omega(\xi_1)-\sum_{j=2}^k d_j G(\xi_1, \xi_j)-\lambda.$$
The second one  $F_2:D_\rho \times (0, +\infty)^{k-1}\times (0, +\infty)\to\mathbb R^{k-1}$
is related to the second $k-1$ equations in \eqref{sys2}:
$$F_2(\boldsymbol\xi, \mathbf{d}, \lambda)=\left(\overline M(\boldsymbol\xi)-\lambda \mathtt{Id}\right)\mathbf{d}^T-\overline G(\boldsymbol\xi)$$
where $$\overline M(\boldsymbol\xi) :=\left(\begin{matrix} \tau_\Omega(\xi_2) & -G(\xi_2, \xi_3) & \ldots & -G(\xi_2, \xi_k)\\
-G(\xi_2, \xi_3) & \tau_\Omega(\xi_3) &\ldots &-G(\xi_3, \xi_k)\\
\ldots &\ldots &\ldots &\ldots\\
\ldots &\ldots &\ldots &\ldots\\
G(\xi_2, \xi_k) & -G(\xi_3, \xi_k) & \ldots &\tau_\Omega(\xi_k)
\end{matrix}\right)\ \hbox{and}\ \overline G(\boldsymbol\xi):=\left(\begin{matrix} G(\xi_1, \xi_2)\\ G(\xi_1, \xi_3) \\
\ldots \\
\ldots \\
G(\xi_1, \xi_k) 
\end{matrix}\right)$$
The third one $F_3:D_\rho \times (0, +\infty)^{k-1}\times (0, +\infty)\to\mathbb R^{4k}$
is related to the last $4k$ equations in \eqref{sys2}:
$$\begin{aligned}F_3 (\boldsymbol\xi, \mathbf{d}, \lambda)= \left(\widetilde M^1(\boldsymbol\xi)\mathbf{d}^T, \widetilde M^2(\boldsymbol\xi)\mathbf{d}^T, \widetilde M^3(\boldsymbol\xi)\mathbf{d}^T, \widetilde M^4(\boldsymbol\xi)\mathbf{d}^T\right)\end{aligned}$$
where
\begin{equation}\label{tildeMxi}\widetilde M^\ell(\boldsymbol\xi)=\(\widetilde m_{ij}^\ell(\xi)\)_{1\leq i, j\leq k}\ \hbox{and}\ \widetilde m_{ij}^\ell(\xi):=\left\{\begin{aligned}&\frac{\partial}{\partial(\xi_i)_\ell}  \tau_\Omega(\xi_i)\quad \ \ \ & \mbox{if}\,\, i=j\\
&-2\frac{\partial}{\partial(\xi_i)_\ell}  G(\xi_i, \xi_j)\quad &\mbox{if}\,\, i\neq j.\\
\end{aligned}\right.\end{equation}
Finally we also define
$$F(\boldsymbol\xi, \mathbf{d}, \lambda) :=(F_1(\boldsymbol\xi, \mathbf{d}, \lambda) , F_2(\boldsymbol\xi, \mathbf{d}, \lambda), F_3(\boldsymbol\xi, \mathbf{d}, \lambda) ).$$
It is clear that  solving \eqref{sys2} is equivalent to finding   $\boldsymbol\xi=\boldsymbol\xi(\varepsilon) \in D_\rho$,  $\lambda=\lambda(\varepsilon)\in (0, +\infty)$ and  $\mathbf{d}=\mathbf d(\varepsilon)=(d_2, \ldots, d_k)\in (0, +\infty)^{k-1}$ so that
\begin{equation}\label{globale}F(\boldsymbol\xi , \mathbf{d} , \lambda ) +o(1)=0,\end{equation}
where $o(1)$ is a continuous function in $(\boldsymbol\xi, \mathbf{d}, \lambda)$ which converges toward zero as $\varepsilon\to0$ 
uniformly  in compact sets of $D_\rho \times (0, +\infty)^{k-1}\times (0, +\infty).$
\\
We shall use  Lemma \ref{isolated} to prove that if $\mathscr K$ is a stable critical set of $\Lambda_1(\boldsymbol \xi)$ in the sense of Definition \eqref{yy1} then there exists a set $\mathscr Z$ such that
$F(\boldsymbol\xi,\mathbf{d},\lambda)=0$ for any $(\boldsymbol\xi,\mathbf{d},\lambda)\in \mathscr Z$ and 
 $\mathtt{deg}\left(F,\Xi,0\right)\not=0$ for an open neighbourhood of $\mathscr Z$.  By the properties of Brouwer degree, we immediately deduce the existence as $\varepsilon$ is small enough of $(\boldsymbol\xi (\varepsilon),\mathbf{d} (\varepsilon),\lambda (\varepsilon))$  close to $\mathscr Z$
solution of \eqref{globale}.\\

First of all, we observe that for any $\boldsymbol\xi\in D_\rho$  there exists a unique $\mathbf{d}=\mathbf{d}(\xi)\in(0, +\infty)^{k-1}$ and
$\lambda=\lambda(\xi)\in (0,+\infty)$ such that 
$$F_1(\boldsymbol\xi, \mathbf{d}(\boldsymbol\xi), \lambda(\boldsymbol\xi))=0\ \hbox{and}\ F_2(\boldsymbol\xi, \mathbf{d}(\boldsymbol\xi), \lambda(\boldsymbol\xi))=0.$$
Indeed, those equations are equivalent to the system
$$\left\{\begin{aligned}& \tau_\Omega(\xi_1)-\sum_{j=2}^k d_j G(\xi_1, \xi_j)-\lambda=0\\
 &\left(\overline M(\boldsymbol\xi)-\lambda \mathtt {Id}\right)\mathbf{d}^T-\overline G(\boldsymbol\xi)=0.\end{aligned}\right.$$
It is immediate to check that it has the solution
\begin{equation}\label{solus}\lambda(\xi):=\Lambda_1( \boldsymbol\xi)\quad \hbox{and}\quad  \mathbf d(\xi)^T:=\left( \overline M(\boldsymbol\xi)-\Lambda_1( \boldsymbol\xi) \mathtt {Id} \right)^{-1}\overline G(\boldsymbol\xi).\end{equation}
Indeed by the first equation we deduce
$$\lambda=\tau_\Omega(\xi_1)-\sum_{j=2}^k d_j G(\xi_1, \xi_j)$$
and combining the two equations we get
\begin{equation}\label{1}M(\boldsymbol\xi)(1, \mathbf{d})^T=\lambda (1, \mathbf{d})^T.\end{equation}
Now if we multiply \eqref{1} by $\ee(\boldsymbol\xi)$ given in \eqref{impM} we get
$$\Lambda_1( \boldsymbol\xi)\langle \ee(\boldsymbol\xi)^T, (1, \mathbf{d})^T\rangle=\langle \ee^T(\boldsymbol\xi), M(\boldsymbol\xi)(1, \mathbf{d})^T\rangle =\lambda\langle \ee(\boldsymbol\xi)^T, (1, \mathbf{d})^T\rangle,$$  which implies that
$\lambda=\Lambda_1( \boldsymbol\xi).$ Moreover,  since $\Lambda_1( \boldsymbol\xi)$ is simple, it is not difficult to check that the  matrix $\overline M(\boldsymbol\xi)-\Lambda_1( \boldsymbol\xi) \mathtt {Id}$ is invertible, i.e. 
\begin{equation}\label{solus3}{\rm det}\left(\overline M(\boldsymbol\xi)-\Lambda_1( \boldsymbol\xi) \mathtt {Id}\right)\neq 0.\end{equation}
Then it follows the existence of  $\mathbf d$ as in \eqref{solus}.

Let us prove \eqref{solus3}.
Assume by contradiction that there exists $\bar {\mathtt e}\in\mathbb R^{k-1}$ such that
$$\left(\overline M(\boldsymbol\xi)-\Lambda_1( \boldsymbol\xi) \mathtt {Id}\right)\bar {\mathtt e}=0.$$
We claim that $(0,\bar {\mathtt e}^T)\in\mathbb R^{k}$ is an eigenvector of $M(\boldsymbol\xi)$ associated with the eigenvalue $\Lambda_1( \boldsymbol\xi)$ which is not a multiple of ${\mathtt e}(\xi)$ (see \eqref{impM}) and a contradiction arises.
Indeed, 
$$\(M(\boldsymbol\xi)-\Lambda_1( \boldsymbol\xi) \mathtt {Id}\right)(0,\bar {\mathtt e})=(-\sum_{i=2}^k G(\xi_1,\xi_i)\bar e_{i-1},0),$$
the eigenvector  ${\mathtt e}(\xi)=(1,\tilde{\mathtt e})\in\mathbb R\times \mathbb R^{k-1}$ satisfies
$$ 
\(\overline M(\boldsymbol\xi)-\Lambda_1( \boldsymbol\xi) \mathtt {Id}\right)\tilde {\mathtt e}=-\(G(\xi_1,\xi_2) ,\dots,G(\xi_1,\xi_k)\)^T$$
and so
 $$ -\sum_{i=2}^k G(\xi_1,\xi_i)\bar e_{i-1}=\bar {\mathtt e}^T\(\overline M(\boldsymbol\xi)-\Lambda_1( \boldsymbol\xi) \mathtt {Id}\right)\tilde {\mathtt e}=\tilde {\mathtt e}^T\(\overline M(\boldsymbol\xi)-\Lambda_1( \boldsymbol\xi) \mathtt {Id}\right)\bar {\mathtt e}=0.$$

Next we observe that
\begin{equation}\label{solus2}
D_{ \mathbf{d},\lambda}\left(F_1(\boldsymbol\xi, \mathbf{d}, \lambda) , F_2(\boldsymbol\xi, \mathbf{d}, \lambda)\)_{ \mathbf{d}= \mathbf{d}(\xi),\lambda=\lambda(\xi)}\quad \hbox{is invertible.}
\end{equation}
Indeed
$$D_{ \mathbf{d},\lambda}\left(F_1(\boldsymbol\xi, \mathbf{d}, \lambda) , F_2(\boldsymbol\xi, \mathbf{d}, \lambda)\)_{ \mathbf{d}= \mathbf{d}(\xi),\lambda=\lambda(\xi)}=\(\begin{matrix}
&-\overline G(\boldsymbol\xi)^T &-1\\
&\overline M(\boldsymbol\xi)-\Lambda_1(\boldsymbol\xi) \mathtt {Id} &-\mathbf{d}(\boldsymbol\xi)^T\\
\end{matrix}\)$$
and by Schur complement formula
$$\begin{aligned}&\mathtt{det}\ \(\begin{matrix}
&-\overline G(\boldsymbol\xi)^T &-1\\
&\overline M(\boldsymbol\xi)-\Lambda_1(\boldsymbol\xi) \mathtt {Id} &-\mathbf{d}(\boldsymbol\xi)^T\\
\end{matrix}\)\\
&=-\mathtt{det}\(\overline M(\boldsymbol\xi)-\Lambda_1(\boldsymbol\xi) \mathtt {Id}\)\(1+\overline G(\boldsymbol\xi)^T\(\overline M(\boldsymbol\xi)-\Lambda_1(\boldsymbol\xi) \mathtt {Id}\)^{-1}\mathbf{d}(\boldsymbol\xi)^T\)\\
&=-\mathtt{det}\(\overline M(\boldsymbol\xi)-\Lambda_1(\boldsymbol\xi) \mathtt {Id}\)\(1+\mathbf{d}(\boldsymbol\xi)\mathbf{d}(\boldsymbol\xi)^T\)\not=0,\end{aligned}$$
because \eqref{solus3} holds true and 
$$ \left(\overline M(\boldsymbol\xi)-\Lambda_1(\boldsymbol\xi) \mathtt {Id}\right)\mathbf{d}(\boldsymbol\xi)^T=\overline G(\boldsymbol\xi)\ \Rightarrow\
 \overline G(\boldsymbol\xi)^T=\mathbf{d}(\boldsymbol\xi)\left(\overline M(\boldsymbol\xi)-\Lambda_1(\boldsymbol\xi) \mathtt {Id}\right)^T=\mathbf{d}(\boldsymbol\xi)\left(\overline M(\boldsymbol\xi)-\Lambda_1(\boldsymbol\xi) \mathtt {Id}\right).$$
Finally,  we claim that
\begin{equation}\label{var2}
\varphi(\boldsymbol\xi):=F_3(\boldsymbol\xi, \mathbf{d}(\boldsymbol\xi), \Lambda_1(\boldsymbol\xi))=\nabla \Lambda_1(\boldsymbol\xi).
\end{equation}

From \eqref{1} it follows that
$$\begin{aligned}\Lambda_1( \boldsymbol\xi)&:=\frac{(1,\mathbf{d}) M(\boldsymbol\xi)(1,\mathbf d)^T}{\left(1+\sum_{j=2}^k d_{j}^2\right)}\\
&=\frac{\tau_\Omega(\xi_1)+\sum_{j=2}^k d_{j}^2 \tau_\Omega(\xi_j) -2\sum_{j=2}^k d_{j} G(\xi_1, \xi_j)-\sum_{j=2}^k\sum_{\ell\neq 1\atop \ell \neq j}d_{j}d_{\ell} G(\xi_j, \xi_\ell)}{\left(1+\sum_{j=2}^k d_{j}^2\right)}.\end{aligned}$$
Now we compute $\nabla \Lambda_1 $.

$$\begin{aligned} &\left(1+\sum_{j=2}^k d_{j}^2\right)\frac{\partial}{\partial (\xi_1)_1}\Lambda_1\\
&=\frac{\partial}{\partial (\xi_1)_1} \tau_\Omega(\xi_1) -2\sum_{j=2}^k d_{j}\frac{\partial}{\partial (\xi_1)_1} G(\xi_1, \xi_j)\\
& +\left[\sum_{j=2}^k \left(2d_{j}\frac{\partial}{\partial (\xi_1)_1} d_{j} \tau_\Omega(\xi_j)-2\frac{\partial}{\partial (\xi_1)_1} d_{j} G(\xi_1, \xi_j)\right.\right.\\
&\left.\left.\qquad-\sum_{\ell\neq 1\atop \ell\neq j}\left(\frac{\partial}{\partial (\xi_1)_1} d_{j} d_{\ell} G(\xi_\ell, \xi_j)+d_{j,}\frac{\partial}{\partial (\xi_1)_1} d_{\ell} G(\xi_\ell, \xi_j)\right)\right)\right]\\
&-\Lambda_1\sum_{j=2}^k2d_{j} \frac{\partial}{\partial (\xi_1)_1} d_{j}\\
&=\frac{\partial}{\partial (\xi_1)_1} \tau_\Omega(\xi_1) -2\sum_{j=2}^k d_{j}\frac{\partial}{\partial (\xi_1)_1} G(\xi_1, \xi_j)\\
&+\left[\sum_{j=2}^k 2\frac{\partial}{\partial (\xi_1)_1} d_{j} \left(d_{j}\tau_\Omega(\xi_j)- G(\xi_1, \xi_j)-\sum_{\ell\neq 1\atop \ell\neq j} d_{\ell} G(\xi_\ell, \xi_j)\right)\right]-\Lambda_1\sum_{j=2}^k2d_{j} \frac{\partial}{\partial (\xi_1)_1} d_{j}\\
&=\frac{\partial}{\partial (\xi_1)_1} \tau_\Omega(\xi_1) -2\sum_{j=2}^k d_{j}\frac{\partial}{\partial (\xi_1)_1} G(\xi_1, \xi_j)+\left[\sum_{j=2}^k 2\frac{\partial}{\partial (\xi_1)_1} d_{j} d_{j} \Lambda_1\right]-\Lambda_1\sum_{j=2}^k2d_{j} \frac{\partial}{\partial (\xi_1)_1} d_{j}\\
&=\frac{\partial}{\partial (\xi_1)_1} \tau_\Omega(\xi_1) -2\sum_{j=2}^k d_{j}\frac{\partial}{\partial (\xi_1)_1} G(\xi_1, \xi_j)=\left(\widetilde M^1(\xi)\mathbf d^T\right)_{1}
\end{aligned}$$
since for any $i$ $$d_{i}\tau_\Omega(\xi_i) -G(\xi_1, \xi_i)-\sum_{j\neq i\atop j\neq 1} d_{j} G(\xi_i, \xi_j)=\Lambda_1 d_{i}$$ 
Similarly, a direct computations  shows that 
$$\left(1+\sum_{j=2}^k d_{j}^2\right)\frac{\partial}{\partial (\xi_j)_\ell}\Lambda_1(M(\xi))=d_j\left(\widetilde M^\ell(\xi)\mathbf d^T\right)_j.$$
Hence \eqref{var2} follows.

Now, if 
  $\mathscr K  $ is a stable critical set of $\Lambda_1,$ then by Definition \ref{yy1} and \eqref{var2} $\mathtt {deg}(\varphi, \Theta,0)\not=0$ for some neighbourhood $\Theta$ of $\mathscr K.$ Therefore, by 
 Lemma \ref{isolated} the set  $\mathscr Z:=\left\{(\boldsymbol\xi, \mathbf d(\xi), \Lambda_1(\boldsymbol\xi))\ :\ \boldsymbol\xi\in \mathscr K\right\}$ is such that 
$\mathtt {deg}(F, \Xi, 0)\not=0,$  for some neighbourhood $\Xi$ of $\mathscr Z.$ This completes the proof.

\subsection{A key lemma}
Let $\mathcal A\times \mathcal B$ an open subset of $\mathbb R^m\times \mathbb R^h$ and define the the $C^1-$ maps $$\Phi_1: \mathcal A\times \mathcal B \to \mathbb R^h,\quad \quad \Phi_2: \mathcal A\times \mathcal B\to \mathbb R^m.$$ Let, also, for every $(x, y)\in\mathcal A\times\mathcal B$ \begin{equation}\label{F}\Phi(x, y)=(\Phi_1(x, y), \Phi_2(x, y)).\end{equation}

The following result holds.
\begin{lemma}\label{isolated}
Assume for any $x\in\mathcal A$ there exists a unique $y:=y(x)\in\mathcal B$ such that $\Phi_1(x, y(x))=0$ and $D_y \Phi_1(x, y(x))$ is invertible. Let $\varphi(x):=\Phi_2(x, y(x)).$\\ 
If $\mathscr K\subset\mathcal A$ is such that 
$$\varphi(x_0)=0\ \hbox{for any}\ x_0\in\mathscr K\ \hbox{and}\ \mathtt{deg}\left(\varphi, \Theta,0\right)\neq 0$$
for some open neighbourhood $\Theta$ of $\mathscr K$,  then the
set
$$\mathscr Z:=\{(x_0, y_0)\in \mathcal A\times\mathcal B\ :\ x_0\in\mathscr K,\ y_0=y(x_0)\}$$ is such that 
$$\Phi(x_0, y_0)=0\ \hbox{for any}\ (x_0, y_0)\in \mathscr Z\ \hbox{and}\ \mathtt{deg}\left(\Phi, \Xi,0\right)\neq 0,$$
for some open neighbourhood $\Xi$ of $\mathscr Z.$
\end{lemma}
\begin{proof}
First of all, we observe that $\Phi(x_0, y_0)=0$ for any $(x_0, y_0)\in \mathscr Z$, since $y_0=y(x_0)$ and
$$\Phi_1(x_0,y_0)=0\ \hbox{and}\ \Phi_2(x_0,y_0)=\varphi(x_0)=0.$$
Next,  we prove that
 there exists $\rho_0>0$ such that  if $\rho\in(0,\rho_0)$, $x\in \Theta$ and $|y-y(x)|=\rho$  then $\Phi_1(x,y)\not=0.$ 
 Assume by contradiction there exist sequences $\rho_n\to0,$ $x_n\to x_0$ and $y_n\in\mathcal B$ such that
 $ |y_n-\tilde y_n|=\rho_n,$ with $\tilde y_n=y(x_n)\to \tilde y_0=y(x_0) ,$ and  $\Phi_1(x_n,y_n)=0.$ Then by Taylor's formula
 $$\Phi_1(x_n,y_n)=\Phi_1(x_n,\tilde y_n)+D_y\Phi_1(x_n,\tilde y_n)(y_n-\tilde y_n)+\mathcal O\(|y_n-\tilde y_n)|^2\),$$
 we deduce
 $$D_y\Phi_1(x_n,\tilde y_n){y_n-\tilde y_n \over \rho_n}=\mathcal O\(\rho_n\),$$
so passing to the limit, taking into account that (up to a subsequence) ${y_n-\tilde y_n \over \rho_n}\to z$ with $|z|=1$, we get
 $D_y\Phi_1(x_0,\tilde y_0)z=0$ and  a contradiction arises since   $D_y\Phi_1(x_0,\tilde y_0)$ is invertible.
Now, we can define the admissible neighbourhood  of $\mathscr Z$ as
$$\Xi:=\{(x,y)\in\mathcal A\times\mathcal B\ :\ x\in \Theta,\ |y-y(x)|< \rho \}$$
and it is immediate to check that 
 $\Phi(x,y)\not=0$ for any $(x,y)\in\partial \Xi.$ Indeed if  $|y-y(x)|=\rho$ then   $\Phi_1(x,y)\not=0$ by the above remark. On the other hand if $x\in\partial \Theta$ and $\Phi_1(x,y)=0$ then $y=y(x)$ and $\Phi_2(x,y)=\varphi(x)\not=0.$
\\
Finally, to derive $\mathtt{deg}\left(\Phi, \Xi,0\right)$ we need to compute $D\Phi$. First, we point out 
\begin{equation}\label{Dvarphi}\begin{aligned}D_x\varphi&=D_x \Phi_2(x, y(x))+D_y\Phi_2(x, y(x))D_x y \\
&=D_x \Phi_2(x, y(x))-D_y\Phi_2(x, y(x))\circ\left(D_y \Phi_1(x, y(x))\right)^{-1}\circ D_x \Phi_1(x, y(x)),
\end{aligned}\end{equation}
because  by $\Phi_1(x, y(x))=0$  and the invertibility $D_y \Phi_1(x, y(x))$  
$$D_x \Phi_1(x, y(x)) +D_y \Phi_1(x, y(x)) D_x y=0\ \Rightarrow\ D_x y=-\left(D_y \Phi_1(x, y(x))\right)^{-1}\circ D_x \Phi_1(x, y(x)).$$ 
Next, we have
$$D\Phi=\left(\begin{matrix}  D_x \Phi_1& D_y \Phi_1 \\  D_x \Phi_2 & D_y \Phi_2 \end{matrix}\right)$$
and by \eqref{Dvarphi} and  Schur complement formula
$$\begin{aligned}\mathtt{det}(D\Phi)&=\mathtt{det}(D_y \Phi_1)\mathtt{det}\left(-D_x \Phi_2+D_y\Phi_2\circ \left(D_y \Phi_1\right)^{-1}\circ D_x \Phi_1\right)\\
&=\mathtt{det}(D_y \Phi_1)\mathtt{det}(D_x \varphi).\end{aligned}$$ 
The claim follows by the definition of  the Brouwer degree. 
\end{proof}

\section{Examples}\label{sec3}

\subsection{The dumb-bell}
Let us consider $k$  smooth bounded domains $\omega_1,\dots,\omega_k$ with  $\overline{\omega}_i \cap \overline{\omega}_j= \emptyset$ if $i\not=j.$
Given $\rho>0$ small enough we consider the domain $\Omega_\rho$ obtained joining the $k$ domains with thin channels of size $\rho.$  Set $\Omega_0:=\omega_1\cup\dots\cup\omega_k.$
Given a point $\boldsymbol\xi=(\xi_1,\dots,\xi_k)\in \Omega_\rho^k$ we define the matrix
$$M_\rho(\boldsymbol\xi)=\left(\begin{matrix}
&\tau_{\rho}(\xi_1)&-G_\rho(\xi_1,\xi_2)&\dots&-G_\rho(\xi_1,\xi_k)\\
&-G_\rho(\xi_1,\xi_2)&\tau_{\rho}(\xi_2)&\dots&-G_\rho(\xi_2,\xi_k)\\
&\vdots&\vdots&\ddots&\vdots\\
&-G_\rho(\xi_1,\xi_k)&-G_\rho(\xi_2,\xi_k)&\dots&\tau_{\rho}(\xi_k)\\
\end{matrix}\right)$$
where $G_\rho$ is the Green's function and $\tau_\rho$ is the Robin's function associated with  the domain $\Omega_\rho$.  We agree that
$G_0$ is the Green's function and $\tau_0$ is the Robin's function associated with  the domain $\Omega_0$. In particular, we observe that
$$G_0(\xi_i,\xi_j)=0\ \hbox{if}\ \xi_i\in\omega_i,\ \xi_j\in\omega_j,\ i\not=j\ \hbox{and}\ \tau_0(\xi_i)=\tau_{\omega_i}(\xi_i)\  \hbox{if}\ \xi_i\in\omega_i.$$ We introduce the matrix
$$M_0 (\boldsymbol\xi):=\left(\begin{matrix}
&\tau_{\omega_1}(\xi_1)&0&\dots&0\\
&0&\tau_{\omega_2}(\xi_2)&\dots&0\\
&\vdots&\vdots&\ddots&\vdots\\
&0&0&\dots&\tau_{\omega_k}(\xi_k)\\
\end{matrix}\right),\ \xi_1\in\omega_1,\dots,\xi_k\in\omega_k.$$

\begin{proposition}
Assume that 
\begin{equation}\label{mini} \min\limits_{\omega_1}\tau_{\omega_1}<\min\left\{\min\limits_{\omega_2}\tau_{\omega_2},\dots,\min\limits_{\omega_{k}}\tau_{\omega_k}\right\}.\end{equation}
Then if $\rho$ is small enough the first eigenvalue of the matrix $M_\rho$ has a stable minimum set.
\end{proposition}
\begin{proof}
First of all, we remind that (see  \cite[Lemma 3.2]{MP})
$$\lim_{\rho \to 0} \tau_{\rho}(\xi)= \tau_{0}(\xi)\quad \hbox{$C^1$-uniformly on compact sets of $\Omega_0$}
$$
and
$$\lim_{\rho \to 0} G_{\rho}(\xi,\eta) = G_{\omega_i}(\xi,\eta)\quad \hbox{$C^1-$ uniformly on compact sets of $\Omega_0\times\Omega_0\setminus\{\xi=\eta\},$}
$$
which imply  
\begin{equation}\label{mrho}\lim\limits_{\rho\to0} M_\rho(\boldsymbol\xi)=M_0 (\boldsymbol\xi)\ \hbox{
$C^1-$ uniformly on compact sets of $\omega_1\times\omega_k$.}\end{equation}
We will prove that the first eigenvalue $\Lambda_1(\boldsymbol\xi)$ of the matrix $M_0(\boldsymbol\xi)$ has a strict local minimum set in $\omega_1\times\dots\times\omega_k$. The claim will follow immediately by \eqref{mrho} and Definition \ref{yyl}.\\
Now, we observe that by \eqref{mini} the minimum value of $\Lambda_1$ is
$$\min\limits_{ \omega_1\times\dots\times\omega_k} \Lambda_1 =\min\limits_{\omega_1}\tau_{\omega_1}$$
and 
 is achieved on the set 
 $$\mathscr K:=\left\{(\xi_1,\xi_2,\dots,\xi_k)\in \omega_1\times\dots\times\omega_k\  :\ \min\limits_{\omega_1}\tau_{\omega_1}=\tau_{\omega_1}(\xi_1),\ \xi_i\in\omega_i,\ i=2,\dots,k\right\}.$$
Next,
we consider the open set
$$\Theta=\left\{\boldsymbol\xi=(\xi_1,\dots,\xi_k)\in \omega_1\times\dots\times\omega_k\ :\ \tau_{\omega_1}(\xi_1)<
\min\left\{\tau_{\omega_2}(\xi_2),\dots,\tau_{\omega_k}(\xi_k)\right\} \right\}.$$
Firstly, we point out that $\Theta$ is a neighbourhood of $\mathscr K,$ i.e. $\mathscr K\subset \Theta$. Indeed by 
$$\min\limits_{\omega_{i}}\tau_{\omega_i}\leq \tau_{\omega_i}(\xi_i)\ \hbox{for any}\ \xi_i\in\omega_i$$
we deduce
$$\min\left\{\min\limits_{\omega_2}\tau_{\omega_2},\dots,\min\limits_{\omega_{k}}\tau_{\omega_k}\right\}\leq \min\left\{\tau_{\omega_2}(\xi_2),\dots,\tau_{\omega_k}(\xi_k)\right\}\ \hbox{for any}\ \xi_i\in\omega_i.$$
Therefore if $(\xi_1,\xi_2,\dots,\xi_k)\in \mathscr K,$  since  $\tau_{\omega_1}(\xi_1)=\min\limits_{\omega_1}\tau_{\omega_1}$, by \eqref{mini}
$$\tau_{\omega_1}(\xi_1)=\min\limits_{\omega_1}\tau_{\omega_1}<\min\left\{\min\limits_{\omega_2}\tau_{\omega_2},\dots,\min\limits_{\omega_{k}}\tau_{\omega_k}\right\} \leq  \min\left\{\tau_{\omega_2}(\xi_2),\dots,\tau_{\omega_k}(\xi_k)\right\},$$
that is $(\xi_1,\xi_2,\dots,\xi_k)\in D.$
Next, we  observe that if $\boldsymbol\xi\in\partial \Theta$ by \eqref{mini}
$$\begin{aligned}\Lambda_1(\boldsymbol\xi)&= \min\limits_{i=1,\dots,k} \tau_{\omega_i}(\xi_i)=\tau_{\omega_1}(\xi_1)= \min\left\{\tau_{\omega_2}(\xi_2),\dots,\tau_{\omega_k}(\xi_k)\right\}\\
& \geq \min\left\{\min\limits_{\omega_2}\tau_{\omega_2},\dots,\min\limits_{\omega_{k}}\tau_{\omega_k}\right\}>\min\limits_{\omega_1}\tau_{\omega_1}.\end{aligned}$$
On the other hand, it is clear that
$$\min\limits_{\boldsymbol\xi\in  \Theta}\Lambda_1(\boldsymbol\xi)=\min\limits_{\omega_1}\tau_{\omega_1}=\min\limits_{ \omega_1\times\dots\times\omega_k} \Lambda_1.$$
Therefore, $\mathscr K$ is a strict local minimum set of $\Lambda_1$. That concludes the proof.
\end{proof}
\subsection{The  thin  annulus}
Let us consider  the annulus
\begin{equation*}
\Omega_{\rho} = \{ x \in \mathbb{R}^{4} : \rho<|x|<1 \},\ 0<\rho<1.
\end{equation*} 
Given an integer $k$, taking into account the symmetries of the annulus, 
  we can look for solutions which are even with respect to $x_3$ and $x_4$ and   invariant  with respect to a rotation of an angle $\frac{2\pi}{k}$
  in the $(x_1,x_2)-$plane (see \cite{dfm2,prey}) and concentrate at points
\begin{equation*}
\xi_{j}(r) =(re^{2\pi \mathtt i \frac{j-1}{k}},0) \in \mathbb{R}^{2} \times \mathbb{R}^{2} \ \ \ \ j=0, \cdots k-1.
\end{equation*}
 Define the matrix $$\mathtt{M}(r) = M(\xi(r))$$ where $\xi(r)=(\xi_1(r), \cdots, \xi_{k}(r))$ and denote by $\Lambda_{j}(r)$ the eigenvalues of $\mathtt{M}(r)$ being $\Lambda_{1}(r)$ the smallest one. It is clear that we are lead to find stable critical points of the one-variable function $r\to\Lambda_1(r),$ $r\in(\rho,1).$\\
We note that the matrix $\mathtt{M}(r)$ is a circulant symmetric matrix (see \cite{MS}), i.e. each column is obtained from the previous one by a rotation in the components:
\begin{center}
$A= \left( \begin{matrix}
a_0 & a_{k-1} & a_{k-2} & \cdots & a_2 & a_1 \\
a_1 & a_0 & a_{k-1} & \cdots & a_3 & a_2 \\
a_2 & a_1 & a_0 & \cdots & a_4 & a_3 \\
\vdots & \vdots & \vdots & & \vdots & \vdots \\
a_{k-1} & a_{k-2} & a_{k-3} & \cdots & a_1 & a_0 
\end{matrix} \right)$
\end{center}
Indeed,  the Robin function and   the Green function in an annulus are explicitly given in \cite{GV}, namely
\begin{equation}\label{Greenanello}
G_\rho(x, y):=\frac{1}{2\omega_3|x-y|^2}-\frac{1}{\omega_3}\sum_{m=0}^{\infty} Q_m(x, y) Z_m\left(\frac{x}{|x|}, \frac{y}{|y|}\right)\end{equation}
and
\begin{equation}\label{robinanello}
\tau_{\rho}(x):=\frac{1}{\omega_3}\sum_{m=0}^{\infty}d_mQ_m(|x|).\end{equation}
Here $d_m:=(m+1)^2$,
\begin{equation}\label{Qmxy}
Q_m(x, y):=\frac{\rho^{2m+2} - \rho^{2m+2}(|x|^{2m+2}+|y|^{2m+2}) + |x|^{2m+2}|y|^{2m+2}}{(2m+2)(|x||y|)^{m+2}(1-\rho^{2m+2}),}\end{equation}
\begin{equation}\label{Qmx}
Q_m(|x|):=\frac{\rho^{2m+2} - 2\rho^{2m+2}|x|^{2m+2} + |x|^{4m+4}}{(2m+2)|x|^{2m+4}(1-\rho^{2m+2})}\end{equation}
and
$Z_m(\zeta, \eta)$ represents the zonal harmonics of degree $m$ which have a particularly simple expression in terms of the Gegenbauer (or ultraspherical) polinomials $P_{m}^{\lambda}$. The latter can be defined in terms of generating functions. If we write (see \cite{SW} p. 148)
\begin{equation*}
(1-2rt+r^{2})^{-\lambda} = \sum_{m=0}^{\infty} P_m^{\lambda}(t)r^{m},
\end{equation*}
where $0 \leq r<1$, $|t| \leq 1$ and $\lambda>0$, then the coefficient $P_m^{\lambda}$ is called Gegenbauer polinomial of degree $m$ associated with $\lambda$. \ \\
Furthermore, from Theorem 2.1 of \cite{GV}, we have that for any $\zeta$ and $\eta$ such that $|\zeta|=1$ and $|\eta|=1$ it holds
\begin{equation*}
Z_{m}(\zeta,\eta) = (m+1)P_{m}^{1}(\zeta\cdot\eta).
\end{equation*}
It is also known that $$P_m^1(1)=(m+1),\quad P_m^1(-x)=(-1)^mP_m^1(x).$$
Moreover $Z_{m}$ is invariant by rotation.\\\\ Hence $$\tau_\rho(\xi_j(r))=\tau_\rho(\xi_1(r))$$ since 
$|\xi_j(r)|=r$ for all $j$. Moreover, we have that $G_\rho(\xi_{l}(r), \xi_{j}(r)) = G_\rho(\xi_{l+1}(r), \xi_{j+1}(r))$. \\ Hence we let
$$a_0 = \tau_\rho(\xi_{1}(r)),\quad\hbox{ and}\quad 
a_{j} = - G_\rho(\xi_{1}(r), \xi_{j+1}(r)),\,\ j=1, \cdots k-1.
$$
Moreover, it easily follows that $$a_{k-j}:=-G_\rho(\xi_{1}(r), \xi_{k-j}(r))=-G_\rho(\xi_{1}(r), \xi_j(r))$$
since $\xi_j(r):=r\left(\cos\frac{2\pi j}{k}, \sin\frac{2\pi j}{k}, 0, 0 \right)$ and $$\xi_{k-j}(r):=r\left(\cos\frac{2\pi (k-j)}{k}, \sin\frac{2\pi (k-j)j}{k}, 0, 0 \right)=r\left(\cos\frac{2\pi j}{k}, -\sin\frac{2\pi j}{k}, 0, 0 \right)$$ and the scalar product between $\xi_1(r)=r(1, 0, 0, 0)$ and $\xi_j$ and $\xi_{k-j}$ is equal.\\
Since $\mathtt M(r)$ is a circulant symmetric matrix, it is known that its eigenvalues  are  
\begin{equation*}
 \Lambda_{\ell}(r) = \sum_{j=0}^{k-1} a_j e^{\frac{2\pi \mathtt i}{k} j(\ell-1)}, \ \ell=1, \cdots, k.
\end{equation*}
We also remark that, since $\mathtt{M}(r)$ is symmetric, all the eigenvalues are reals. \ We claim that
$$\Lambda_1(r)=\tau_\rho(\xi_1(r))-\sum_{j=1}^{k-1}G_\rho(\xi_1(r), \xi_{j+1}(r)) $$
is simple.
Indeed
$$\begin{aligned} \Lambda_\ell(r)&=a_0 +\sum_{j=1}^{k-1}a_j e^{\frac{2\pi\mathtt i}{k}j(\ell-1)}=\tau_\rho(\xi_1(r))-\sum_{j=1}^{k-1}G_\rho(\xi_1(r), \xi_{j+1}(r))\mathtt Re\left(e^{\frac{2\pi\mathtt i}{k}j(\ell-1)}\right)\\
&>\tau_\rho(\xi_1(r))-\sum_{j=1}^{k-1}G_\rho(\xi_1(r), \xi_{j+1}(r))=\Lambda_1(r).\end{aligned}$$
It is immediate to check that for any $\rho>0$ there exists a strict minimum point $r(\rho)\in (\rho,1)$ of the function $\Lambda_1,$
since
$$\lim\limits_{r\to \rho}\Lambda_1(r)=\lim\limits_{r\to 1}\Lambda_1(r)=+\infty.$$
We only need to check that 
\begin{equation}\label{c1}\Lambda_1(r_\rho):=\min\limits_{r\in(\rho,1)}\Lambda_1(\rho)>0.\end{equation}
If \eqref{c1} is satisfied, by Theorem \ref{main}, we deduce that  the problem \eqref{pb4} has a positive solution which blow-up and concentrate at $k$ different points  in $\Omega_\rho.$ \\
On the other hand, we are able to prove that \eqref{c1} holds true only when $k=2$ and $k=4$ (see Proposition \ref{thmanello} below). However, we strongly believe that \eqref{c1} is satisfied for any integer $k$ and we also conjecture  that the annulus becomes increasingly thinner as 
$k$ grows, i.e.
\begin{equation}\label{CON}
 \begin{aligned}&\hbox{for any integer $k$ there exists $\rho_k\in(0,1)$ such that for any $\rho\in(\rho_k,1)$}\\
&\hbox{it holds true that }\  \min\limits_{r\in(\rho,1)}\Lambda_1(\rho)>0\ \hbox{and}\  \liminf\limits_{k\to\infty}\rho_k=1.\end{aligned}
\end{equation}
\begin{proposition}\label{thmanello} ${}$
\begin{itemize}
\item If $k=2$, then $\rho_2>\frac{1}{15}$.
\item If $k=4$  then  $\rho_4>\frac{5}{11}$
\end{itemize}\end{proposition}
\begin{proof}
If $k=2$
$$\Lambda_1(r)=\tau_\rho(\xi_1(r))-G_\rho(\xi_1(r), \xi_{2}(r))$$  with
 $\xi_1(r)=r(1, 0, 0, 0)$ while $\xi_2(r)=r(-1, 0, 0, 0)$.\\ 
By using the explicit formula for the Green function in the annulus $\Omega_{a}$ (see \eqref{Greenanello}) and the explicit formula for the Robin function (see \eqref{robinanello}) we get
\begin{equation*}
\tau_\rho(\xi_1(r)) = \frac{1}{\omega_3} \sum_{m=0}^{\infty} d_{m} Q_{m}(r).
\end{equation*}
Instead
\begin{equation*}
G_\rho(\xi_1(r),\xi_2(r)) = \frac{1}{\omega_3} \left[ \frac{1}{8r^{2}} - \sum_{m=0}^{\infty} Q_{m}(r) Z_{m}\left(\frac{\xi_1}{|\xi_1|}, \frac{\xi_2}{|\xi_2|}\right) \right].
\end{equation*}
We get that
\begin{equation*}
Z_{m}\left(\frac{\xi_1}{|\xi_1|}, \frac{\xi_2}{|\xi_2|}\right) = (m+1)P_{m}^{1}(-1)=(-1)^{m}(m+1)P_{m}^{1}(1)=(-1)^{m}(m+1)^{2}=(-1)^m d_m;
\end{equation*}
so
\begin{equation*}
G_\rho(\xi_1(r),\xi_2(r)) = \frac{1}{\omega_3} \left[ \frac{1}{8r^{2}} - \sum_{m=0}^{\infty} (-1)^{m} d_m Q_{m}(r) \right].
\end{equation*}
Notice that $Q_{m}(r)$ is nonnegative for all $m \geq 0$, and therefore,
\begin{equation*}
\begin{aligned}
\tau_\rho(\xi_1(r))-G_\rho(\xi_1(r),\xi_2(r)) & = \frac{1}{\omega_3} \left[ -\frac{1}{8r^{2}} + \sum_{m=0}^{\infty} (d_m+ (-1)^{m} d_m) Q_{m}(r) \right]
\\& = \frac{1}{\omega_3} \left[ -\frac{1}{8r^{2}} + 2Q_{0}(r) + \sum_{m=1}^{\infty} (d_m+ (-1)^{m} d_m) Q_{m}(r) \right]
\\& \geq \frac{1}{\omega_3} \left[ - \frac{1}{8r^{2}} + 2Q_{0}(r) \right].\end{aligned}
\end{equation*}
A sufficient condition to have \eqref{c1} is 
\begin{equation*}
\frac{8\rho^2 -16\rho^2r^2+8r^4}{r^{4}(1-\rho^2)} > \frac{1}{r^2}, \   \forall\ r \in (\rho,1),
\end{equation*}
that is $\rho \in (\frac{1}{15}, 1)$.\\ 
If $k=4$
$$\Lambda_1(r):=\tau_\rho(\xi_1(r))-\sum_{j=1}^3 G_\rho(\xi_1(r), \xi_{j+1}(r))$$ with $$\xi_1(r):=r(1, 0, 0, 0),\,\, \xi_2(r):=r(0, 1, 0, 0), \,\, \xi_3(r):=r(-1, 0, 0, 0),\,\, \xi_4(r):=r (0, -1, 0, 0).$$ 
We also have
$$G_\rho(\xi_1(r), \xi_2(r)):=\frac{1}{4\omega_3 r^2},\quad G_\rho(\xi_1(r), \xi_4(r)):=\frac{1}{4\omega_3 r^2}$$
since $Z_m\left(\frac{\xi_1}{|\xi_1|},\frac{\xi_2}{|\xi_2|}\right)=0$ and $Z_m\left(\frac{\xi_1}{|\xi_1|},\frac{\xi_4}{|\xi_4|}\right)=0$.\\
Moreover arguing as above we get
$$G_\rho(\xi_1(r), \xi_3(r)):=\frac{1}{\omega_3}\left[\frac{1}{8r^2}-\sum_{m=0}^\infty (-1)^m d_m Q_m(r)\right].$$ 
Then
\begin{equation*}
\begin{aligned}
\Lambda_1(r) & = \frac{1}{\omega_3} \left[ -\frac{5}{8r^{2}} + \sum_{m=0}^{\infty} (d_m+ (-1)^{m} d_m) Q_{m}(r) \right]
\\& \geq \frac{1}{\omega_3} \left[ - \frac{5}{8r^{2}} + 2Q_{0}(r) \right].\end{aligned}
\end{equation*}
A sufficient condition to have \eqref{c1} is 
\begin{equation*}
\frac{8\rho^2 -16\rho^2r^2+8r^4}{r^{4}(1-\rho^2)} > \frac{5}{r^2}, \ \forall\ r \in (\rho,1),
\end{equation*}
that is $\rho \in (\frac{5}{11}, 1)$.
\end{proof}
\appendix
\section{Some estimates} \label{app}
\subsection{The estimate of the error}
First of all, we give an estimate of the error term $ \mathcal{E}_{\pmb{\d}, \pmb{\xi}}$.
\begin{lemma}
For any $\rho>0$ small enough there exist $\varepsilon_{0}>0$ and $C>0$ such that for any $\xi \in D_\rho$ and for any $\varepsilon, \delta_i \in (0,\varepsilon_{0})$ it holds 
\begin{equation*}
\| \mathcal{E}_{\pmb{\d}, \pmb{\xi}} \| \leq C\( |\pmb{\delta}|^{2}+\e|\pmb{\d}|\).
\end{equation*}
\end{lemma}
\begin{proof}  
First of all, we remark that
$W_{\pmb{\d},\pmb{\xi}}={i}^*_\Omega\left(\sum\limits_{i=1}^k  (\Ui)^3\right).$ Then by a straightforward computation  
		\begin{align*}
\|\mathcal{E}_{\pmb{\d}, \pmb{\xi}} \| &\lesssim\left( \| (\sum_i  P\Ui)^3-\sum_i (\Ui)^3\|_{\frac 43} + \e\sum_i \|P\Ui\|_{\frac 43}\right)\\
&\lesssim  \sum_i \|(P\Ui)^3-(\Ui)^3\|_{\frac 43} + \sum_{j<i} \| (P\Ui)^2 P\Uj \|_{{\frac 43}}\\&+\sum_{j< i<h} \| P\Ui P\Uj P\Uh\|_{{\frac 43}}  + \e\sum_i \|P\Ui\|_{\frac 43}.
		\end{align*}
Now
	\begin{align*}
		\|(P\Ui)^3-(\Ui)^3\| &\lesssim \left(\d_i \|(\Ui)^2\|_{\frac 43} + \d_i^2\|\Ui\|_{\frac 43} + \d_i^3\right) \lesssim \d_i^2 
	\end{align*}
and $\|P\Ui\|_{\frac 43}= \O(\d_i)$.
Now for all $i \neq j$ $$\begin{aligned}
	\int_\Omega |(P\Ui)^2 P\Uj|^{4/3} &\lesssim \left(\d_j^{4/3} \int\limits_{ B(\xi_i,\rho)} |P\Ui|^{8/3} + \d_i^{8/3} \int\limits_{ B(\xi_j,\rho)} |P\Uj|^{4/3} + {\d_i^2\d_j}^{\frac 43} \right) \\&\lesssim c\left(\d_j^{4/3} \int\limits_{ B(\xi_i,\rho)} (\Ui+\O(\d_i))^{8/3} + \d_i^{8/3} \int\limits_{ B(\xi_j,\rho)} (\Uj+\O(\d_j))^{4/3} \right)\\&\lesssim \left(\d_i\d_j\right)^{4/3}
\end{aligned}$$
and similarly, for $i\neq h, i\neq j, h\neq j$ we have \begin{align*}
	\int_\Omega \left( P\Ui P\Uh P\Uj \right)^{4/3} \lesssim (\d_i\d_h\d_k)^{4/3} .
\end{align*} 
	\end{proof}

\subsection{Proof of Proposition \ref{propro}}

\begin{lemma}\label{lemma4.1}
If for every $h=1, \ldots, k$
\begin{equation}\label{cond1}
\int_\Omega \left(-\Delta u_\e -f(u_\e)-\e u_\e\right)\psi^0_{\d_h, \xi_h}\, dx=0
\end{equation}
and for some $\eta_h>0$ and for $l=1,2,3,4,$
\begin{equation}\label{cond2}
\int_{B(\xi_h, \eta_h)} \left(-\Delta u_\e -f(u_\e)-\e u_\e\right)\partial_l u_\e\, dx=0
\end{equation}
then $\mathfrak c^j_i=0$ for every $j=0, 1, \ldots, 4$ and for every $i=1, \ldots, k$.
\end{lemma}
\begin{proof}
From \eqref{cond1} and \eqref{cond2} we have that, for all $h=1, \cdots, k$ and for $l=0, \cdots, 4$
\begin{equation*}
\sum_{i=1, \cdots, k  \atop j=0, \cdots, 4}  \mathfrak c^j_i \int_{\Omega} (\Ui)^{2} \psi^{j}_{i} \psi^{0}_{h} = \sum_{i=1, \cdots, k  \atop j=0, \cdots, 4} \mathfrak c^j_i  \int_{B(\xi_{h}, \eta_{h})} (\Ui)^{2} \psi^{j}_{i} \partial_{l} u_{\e} = 0 
\end{equation*}
and the linear system in the $ \mathfrak c^j_i$'s diagonally dominant. Indeed
\begin{equation}\label{Sprimo}
\int_{\Omega} (\Ui)^{2} \psi^{j}_{i} \psi^{0}_{h} = 
\begin{cases}
a + \mathcal{O}(\d_i^4) \ \ \ \ \ \ \ \ \ \ \ \ \ \ h=i \ \mbox{and} \ j=0, \\
\mathcal{O}(\d_i^5) \ \ \ \ \ \ \ \ \ \ \ \ \ \ \ \ \ \ \ \ \ h=i \ \mbox{and} \ j \neq 0, \\
\mathcal{O}(|\pmb \d|^2) \ \ \ \ \ \ \ \ \ \ \ \ \ \ \ \ \ \ \ h\neq i, \forall j,
\end{cases}
\end{equation}
for some constant $a \neq 0$.
Moreover
\begin{equation*}
\int_{B(\xi_{h}, \eta_{h})} (\Ui)^{2} \psi^{j}_{i} \partial_{l} u_{\e} =
\begin{cases}
\frac{1}{\d_i} b + \mathcal{O}(\e) \ \ \ \ \ \ \ \ \ \ \ \  h=i \ \mbox{and} \ j=l, \\
\mathcal{O}(\e) \ \ \ \ \ \ \ \ \ \ \ \ \ \ \ \ \ \ \ \ \ h=i \ \mbox{and} \ j \neq l, \\
\mathcal{O}(|\pmb \d|^3) \ \ \ \ \ \ \ \ \ \ \ \ \ \ \ \ \ h\neq i, j =0, \\
\mathcal{O}(|\pmb \d|^4) \ \ \ \ \ \ \ \ \ \ \ \ \ \ \ \ \ h\neq i, j =1, \cdots, 4,
\end{cases}
\end{equation*}
for some constant $b \neq 0$, because by \eqref{soluzione} and \eqref{phi},
\begin{equation*}
\int_{B(\xi_{h}, \eta_{h})} (\Ui)^{2} \psi^{j}_{i} \partial_{l} \phi=
\begin{cases}
\underbrace{\mathcal{O}(\| \phi \|) \left(\int_{B(\xi_{h}, \eta_{h})} |(\Ui)^{2} \psi^{j}_{i}|^{2} \right)^{\frac{1}{2}}}_{\mathcal{O}\left(\e \d_i \frac{1}{\d_i} \right)} = \mathcal{O}(\e) \hspace{1 cm} h=i, \ j=0, \cdots, 4, \\
\mathcal{O}(\e|\pmb \d|^4), \hspace{6.5 cm}  h \neq i, \ j=0,\\
\mathcal{O}(\e|\pmb \d|^5), \hspace{6.5 cm} h \neq i, \ j=1, \cdots, 4,
\end{cases}
\end{equation*}
and
\begin{equation*}
\int_{B(\xi_{h}, \eta_{h})} (\Ui)^{2} \psi^{j}_{i} \partial_{l} P\Uh =
\begin{cases}
\frac{1}{\d_i} b + \mathcal{O}(\d_i^2) 	\hspace{3.4 cm} h=i, \ l=j, \\
\mathcal{O}(\d_i^2), 	\hspace{4.2 cm}  h=i, l \neq j,\\
\mathcal{O}(|\pmb \d|^3),  	\hspace{4 cm}  h \neq i, j=0, \\
\mathcal{O}(|\pmb \d|^4), 	\hspace{4 cm}  h \neq i, j=1, \cdots, 4.
\end{cases}
\end{equation*}
\end{proof}
We next establish the main terms of \eqref{cond1} and \eqref{cond2}.\\
\begin{proposition}\label{espansione1}
For any $\rho>0$ small enough 
\begin{equation}\label{sys0}\begin{aligned}
&\int_\Omega \left(-\Delta u_\e -f(u_\e)-\e u_\e\right)\psi^0_{\d_h, \xi_h}\, dx\\ & \qquad=\mathfrak C^2\left[\delta_h^2\tau_\Omega(\xi_h)-\sum_{i\neq h}\d_i\d_h G(\xi_i, \xi_h)+\frac{1}{4\omega}\e \d_h^2\ln\d_h+o(|\pmb\d|^2)\right]\end{aligned}
\end{equation} as $\e \to 0$, uniformly  with respect to $(\boldsymbol\xi, \mathbf{d}, \lambda)$ in compact sets of $D_\rho \times (0, +\infty)^{k-1}\times (0, +\infty).$
\end{proposition}
\begin{proof}
We point out that 
\begin{equation*}
\begin{aligned}
\int_{\Omega} \left(-\Delta u_\e -f(u_\e)-\e u_\e\right) \psi^0_{\d_h, \xi_h}\, dx & = \underbrace{\int_{\Omega} (-\Delta W_{\pmb{\d}, \pmb{\xi}} - W_{\pmb{\d}, \pmb{\xi}}^{3} - \e W_{\pmb{\d}, \pmb{\xi}})\psi^0_{\d_h, \xi_h}\, dx}_{(I_1)} \\& + \underbrace{\int_{\Omega} (-\Delta \phi - 3 \phi W_{\pmb{\d}, \pmb{\xi}}^2 - \e \phi) \psi^0_{\d_h, \xi_h}\, dx}_{(I_{2})} \\& + \underbrace{\int_{\Omega} 
\left(f(W_{\pmb{\d}, \pmb{\xi}}+\phi) -f(W_{\pmb{\d}, \pmb{\xi}}) -f'(W_{\pmb{\d}, \pmb{\xi}})\phi \right) \psi^0_{\d_h, \xi_h}\, dx}_{(I_{3})}
\end{aligned}
\end{equation*}
First of all, let us prove that 
\begin{equation}\label{I1}
(I_1)= \mathfrak{C}^{2} \delta_h^2\tau_\Omega(\xi_h)- \mathfrak{C}^{2} \sum_{i\neq h}\d_i\d_h G(\xi_i, \xi_h)+ \frac{\mathfrak{C}^{2}}{4\omega}\e \d_h^2\ln\d_h+o(|\pmb\d|^2).
\end{equation}
We observe that 
\begin{equation*}
\begin{aligned}
(I_1) & = \int_{\Omega} (\sum_{i=1}^{k} (\Ui)^{3} - (\sum_{i=1}^{k} P\Ui)^{3} - \e \sum_{i=1}^{k}P\Ui) \psi^{0}_{\d_h, \xi_h} \ dx 
\\& = \int_{\Omega} \left[ (\Uh)^{3} - (P\Uh)^{3} \right] \psi^{0}_{\d_h, \xi_h} + \sum_{i \neq h} \int_{\Omega} \left[ (\Ui)^{3} - (P\Ui)^{3} \right] \psi^{0}_{\d_h, \xi_h} 
\\& - 3 \sum_{i \neq h} \int_{\Omega} P\Ui (P\Uh)^{2} \psi^{0}_{\d_h, \xi_h}+ 3 \sum_{i \neq h}( P\Ui)^{2} P\Uh \psi^{0}_{\d_h, \xi_h} 
\\& + 6 \sum_{i,l,m=1}^{k} \int_{\Omega} P\Ui P\Ul P\UM \psi^{0}_{\d_h, \xi_h} - \e \int_{\Omega} P\Uh \psi^{0}_{\d_h, \xi_h} \\& - \e \sum_{i \neq h} \int_{\Omega} P\Ui \psi^{0}_{\d_h, \xi_h}
 \\&
= \int_{\Omega} \left[ (\Uh)^{3} - (P\Uh)^{3} \right] \psi^{0}_{\d_h, \xi_h} - 3 \sum_{i \neq h} \int_{\Omega} P\Ui (P\Uh)^{2} \psi^{0}_{\d_h, \xi_h} - \e \int_{\Omega} P\Uh \psi^{0}_{\d_h, \xi_h} + o(|\pmb\d|^2).
\end{aligned}
\end{equation*}
Now 
\begin{equation*}
\begin{aligned}
\int_{\Omega} \left[ (\Uh)^{3} -(P\Uh)^{3} \right] \psi^0_{\d_h, \xi_h}\, dx & = 6 \alpha_{4} \omega \delta_{h} \int_{\Omega}(\Uh)^2 \psi^0_{\d_h, \xi_h} H(x, \xi_{h}) + \mathcal{O}(\d_i^{3}) \\&= 6 \alpha_{4}^{4} \omega \delta_{h}^{2} H(\xi_{h}, \xi_{h}) \int_{B(0, \frac{\rho}{\d_h})} \frac{|y|^{2}-1}{(1+|y|^{2})^{4}} \ dy + \mathcal{O}(\d_i^{3}) \\& = \mathfrak{C}^{2} \delta_h^2\tau_\Omega(\xi_h) + o(|\pmb\d|^2)
\end{aligned}
\end{equation*}
and, if $i \neq h$,
\begin{equation*}
\begin{aligned}
3 \sum_{i \neq h} \int_{\Omega} P\Ui (P\Uh)^2 \psi^0_{\d_h, \xi_h}\, dx & = 6 \alpha_{4} \omega \delta_{i} \sum_{i \neq h} \int_{\Omega} (\Uh)^2 \psi^0_{\d_h, \xi_h} G(x, \xi_{i})\, dx + o(\d_i \d_h)
\\&= 6 \alpha_{4}^{4} \omega \sum_{i \neq h} \delta_{i} \delta_{h} G(\xi_{i},\xi_{h})  \int_{B(0, \frac{\rho}{\d_h})} \frac{|y|^{2}-1}{(1+|y|^{2})^{4}} \ dy + o(\d_i \d_h) 
\\& = \mathfrak{C}^2 \sum_{i\neq h}\d_i\d_h G(\xi_i, \xi_h) + o(|\pmb \d|^{2}).
\end{aligned}
\end{equation*}
We used that 
\begin{equation*}
\int_{\mathbb{R}^{4}} \frac{|y|^{2}-1}{(1+|y|^{2})^{4}} \ dy = \frac{\omega}{12}.
\end{equation*}
Furthermore
\begin{equation*}
\begin{aligned}
\e \int_{\Omega} P\Uh \psi^0_{\d_h, \xi_h} \ dx & = \e \alpha_4^2 \d_h^2 \int_{B(\xi_h, \rho)} \frac{|x-\xi_h|^{2} - \d_h^2}{(\d_h^2 + |x-\xi_h|^2)^{3}} dx + \mathcal{O}(\d_h^3) \\&
= \e \alpha_{4}^{2} \delta_{h}^{2} \int_{B(0, \frac{\rho}{\delta_{h}})} \frac{|y|^{2}-1}{(1+|y|^{2})^{3}} \ dy + \mathcal{O}(\d_h^3) \\& 
= - \e \frac{\mathfrak{C}^2}{4\omega} \d_h^{2} \ln \d_h + o(\e \d_h^{2}\ln \d_h).
\end{aligned}
\end{equation*}
Let us consider the other terms. It is important to point out the estimate
\begin{equation*}
\int_{\partial \Omega} |\partial_{\nu} \phi|^{2} = o(|\pmb \d|^{2})
\end{equation*}
proved in \cite{R}. Then, recalling that $-\Delta \psi^0_{\d_h, \xi_h} = 3(\Uh)^{2} \psi^0_{\d_h, \xi_h}$, for all $h=1, \cdots, k$, we have
\begin{equation*}
\begin{aligned}
\int_{\Omega} (-\Delta \phi) \psi^0_{\d_h, \xi_h} & = \int_{\Omega} \phi (-\Delta \psi^0_{\d_h, \xi_h}) + \int_{\partial \Omega} \underbrace{\phi}_{=0} \nabla \psi^0_{\d_h, \xi_h} \cdot \nu - \int_{\partial \Omega} \underbrace{\psi^0_{\d_h, \xi_h}}_{=\mathcal{O}(\d_h)} \nabla \phi \cdot \nu \\& = 3\int_{\Omega} \phi (\Uh)^{2} \psi^0_{\d_h, \xi_h} + o(|\pmb \d|^{2}).
\end{aligned}
\end{equation*}
Furthermore, 
\begin{equation}\label{I2}
\begin{aligned}
|(I_2)| & \leq 3 \int_{\Omega} |\phi| | (\Uh)^{2} -  W_{\pmb{\d}, \pmb{\xi}}^2| |\psi^{0}_{\d_h, \xi_h}| + \e \int_{\Omega} |\phi| | \psi^{0}_{\d_h, \xi_h}| + o(|\pmb \d|^{2})
\\& \lesssim \| \phi \|_{4} \left \|\left[ (\Uh)^{2} -  W_{\pmb{\d}, \pmb{\xi}}^2\right] \psi^{0}_{\d_h, \xi_h} \right \|_{\frac{4}{3}} + \e \|\phi\|_{4} \|\psi^{0}_{\d_h, \xi_h}\|_{\frac{4}{3}}.
\end{aligned}
\end{equation}
Moreover, we can observe that
\begin{equation*}
(\Uh)^{2} -  W_{\pmb{\d}, \pmb{\xi}}^2 = (\Uh)^{2} - (P\Uh)^{2} - \sum_{i \neq h} (P\Ui) ^{2} - 2 \sum_{i \neq h} P\Uh P\Ui - 2 \sum_{i,l \neq h} P\Ui P\Ul;
\end{equation*}
so, recalling that $P\Uh, \psi^{0}_{\d_h, \xi_h} \lesssim \d_h$ on $\Omega \setminus B(\xi_h, \rho)$, we have that
\begin{equation}\label{W}
\begin{aligned}
\left \| \left[ (\Uh)^{2} -  W_{\pmb{\d}, \pmb{\xi}}^2 \right] \psi^{0}_{\d_h, \xi_h} \right \|_{\frac{4}{3}} & \lesssim \left \| \left[ (\Uh)^{2} - (P\Uh)^{2} \right] \psi^{0}_{\d_h, \xi_h} \right \|_{\frac{4}{3}} + \sum_{i \neq h} \| (P\Ui)^{2} \psi^{0}_{\d_h, \xi_h} \|_{\frac{4}{3}} \\& + \sum_{i \neq h} \| P\Uh P\Ui \psi^{0}_{\d_h, \xi_h}\|_{\frac{4}{3}} + \sum_{i,l \neq h} \| P\Ui P\Ul \psi^{0}_{\d_h, \xi_h}\|_{\frac{4}{3}} \\& \lesssim \| \left[ (\Uh)^{2} - (P\Uh)^{2} \right] \psi^{0}_{\d_h, \xi_h}\|_{\frac{4}{3}} + \mathcal{O}(|\pmb \d|^{2}) =  \mathcal{O}(|\pmb \d|^{2}).
\end{aligned}
\end{equation}
Combining \eqref{W} and \eqref{phi} with \eqref{I2} we get $(I_2)=o(|\pmb \d|^{2})$.
Furthermore
\begin{equation}\label{I3}
|(I_3)| \lesssim \| \phi^{3} \|_{H^1_0(\Omega)} \| \psi^0_{\d_h, \xi_h} \|_{4} +  \| \phi \|_{H^1_0(\Omega)}^{2} \sum_{i=1}^{k} \| P\Ui \psi^0_{\d_h, \xi_h} \|_{2} = o(|\pmb \d|^{2}).
\end{equation}
Finally, \eqref{sys0} follows by \eqref{I1}, \eqref{I2} and \eqref{I3}.
\end{proof}
\begin{proposition}\label{espansione2}
For every $\rho>0$ small enough there exists $\eta_h>0$ such that 

\begin{equation}\label{sysj}\begin{aligned}
\int_{B(\xi_h, \eta_h)} \left(-\Delta u_\e -f(u_\e)-\e u_\e\right)\partial_l u_\e\, dx &= -\frac 12\mathfrak C^2\frac{\partial}{\partial(\xi_h)_l}\left[\d_h^2\tau_\Omega(\xi_h)-2\sum_{h\neq i}\d_i\d_hG(\xi_i,\xi_h)+\O\(|\pmb\d|^3\)\right]
\end{aligned}
\end{equation}
\end{proposition}
\begin{proof}
First of all, we point out that
\begin{equation}\label{K1}
\int_{B(\xi_h, \eta_h)} \left(-\Delta u_\e -f(u_\e)-\e u_\e\right)\partial_l u_\e\, dx = \int_{\partial B(\xi_h, \eta_h)} \left( -\frac{1}{2} (\partial_{\nu} u_{\e})^{2} \cdot \nu_l - \frac{1}{4} u_\e^{4} \nu_{l} - \frac{1}{2} \e u_\e^{2} \nu_l \right) \ dx.
\end{equation} 
Indeed, in general, 
\begin{equation}\label{deriv}
\begin{aligned}
\int_{\Omega} - \Delta u_\e \cdot \partial_l u_\e & = - \sum_{i=1}^{N} \int_{\Omega}  \frac{\partial^2 u_\e}{\partial y_i^2} \frac{\partial u_\e}{\partial y_l} = \frac 12 \sum_{i=1}^{N} \int_{\Omega} \frac{\partial}{\partial y_l} \left( \frac{\partial u_\e}{\partial y_i} \right)^2 - \sum_{i=1}^{N} \int_{\partial \Omega} \frac{\partial u_\e}{\partial y_l} \frac{\partial u_\e}{\partial y_i} \cdot \nu_i
\\& = \frac 12 \sum_{i=1}^{N} \int_{\partial\Omega} \left( \frac{\partial u_\e}{\partial y_i} \right)^2 \cdot \nu_l - \int_{\partial \Omega} \frac{\partial u_\e}{\partial y_l} \frac{\partial u_\e}{\partial \nu} = -\frac 12 \int_{\partial \Omega} \left( \frac{\partial u_\e}{\partial \nu} \right)^2 \cdot \nu_l.
\end{aligned}
\end{equation}
Then, using \eqref{deriv} and integrating by parts we have \eqref{K1}. \ \\
Moreover, by co-area formula and \eqref{phi}, we choose that 
\begin{equation}\label{K2}
\int_{\partial B(\xi_h, \eta_h)} (|\nabla \phi|^{2} + |\phi|^{4} + \e|\phi|^{2}) \lesssim |\pmb\d|^4.
\end{equation}
Now, by \eqref{soluzione}
\begin{equation}\label{K3}
W_{\pmb{\d}, \pmb{\xi}} := \frac{\alpha_4 \d_h}{|x-\xi_h|^{2}} + 2\alpha_4 \omega \underbrace{\left[-\d_h H(x, \xi_h) + \sum_{h \neq i} \d_i G(x, \xi_i) \right]}_{\Theta_{i,h}(x)} + \mathcal{O}(\d_h^{3})
\end{equation}
$C^{1}$-uniformly on $\partial B(\xi_h, \eta_h)$. It is crucial to point out that the function $\Theta_{i,h}$ is harmonic on the ball $B(\xi_h, \eta_h)$. Therefore, by \eqref{K1}, \eqref{K2} and \eqref{K3}
\begin{equation*}
\begin{aligned}
& \int_{B(\xi_h, \eta_h)} \left(-\Delta u_\e -f(u_\e)-\e u_\e\right)\partial_l u\, dx = - \int_{\partial B(\xi_h, \eta_h)} \frac{1}{2} (\partial_{\nu} u_{\e})^{2} \cdot \nu_l + o(|\pmb\d|^2) 
\\& = - 2\alpha_4^2 \omega^2 \underbrace{\int_{\partial B(\xi_h, \eta_h)} (\partial_{\nu} \Theta_{i,h}(x))^{2} \cdot \nu_l}_{(=\int_{B(\xi_h, \eta_h)} \Delta \Theta_{i,h}(x) \partial_{l} \Theta_{i,h}(x) = 0)} - \alpha_{4}^{2} \d_h^{2} \underbrace{\int_{\partial B(\xi_h, \eta_h)} \partial_{\nu} \frac{1}{|x-\xi_h|^{2}} \cdot \nu_l}_{(=\int_{B(\xi_h, \eta_h)} \Delta  \frac{1}{|x-\xi_h|^{2}} \partial_{l}  \frac{1}{|x-\xi_h|^{2}} =0)} 
\\& - 2\alpha_4 \omega  \int_{\partial B(\xi_h, \eta_h)} \nabla \left( \frac{\alpha_4 \d_h}{|x-\xi_h|^{2}} \right) \cdot \nu \ \partial_{x_l} \Theta_{i,h}(x) + o(|\pmb\d|^2) 
\\& = 4 \alpha_4^2 \omega \d_h \frac{1}{|\eta_h|^{3}} \int_{\partial B(\xi_h, \eta_h)} \partial_{x_l} \Theta_{i,h}(x) +  o(|\pmb\d|^2) = 4 \alpha^2_4 \omega^2 \d_h \partial_{l}  \Theta_{i,h}(\xi_h) + o(|\pmb\d|^2) 
\end{aligned}
\end{equation*} 
because $\Theta_{i,h}$ is harmonic on the ball $B(\xi_h, \eta_h)$ and from the mean value theorem
\begin{equation*}
\frac{1}{|\partial B(\xi_h, \eta_h)|}  \int_{\partial B(\xi_h, \eta_h)} \partial_{x_l} \Theta_{i,h}(x) = \partial_{l} \Theta_{i,h}(\xi_h) \ \ \mbox{with} \ \ |\partial B(\xi_h, \eta_h)| = \omega |\eta_h|^3.
\end{equation*}
Finally, as $\tau_\Omega(x)=H(x,x)$ denotes the Robin's function, by
\begin{equation*}
\partial_{(\xi_h)_{l}} \tau_\Omega(\xi_h) = (\partial_{x_l} H(x,y) + \partial_{y_l} H(x,y))|_{(x,y)=(\xi_h, \xi_h)} = 2\partial_{x_l} H(x,y)|_{(x,y)=(\xi_h, \xi_h)}
\end{equation*}
the claim follows.
\end{proof} 


\end{document}